\documentclass[11pt, reqno]{amsart}
\usepackage[a4paper, left=1in, right=1in, top=1in, bottom=1in]{geometry}
\usepackage{amsmath,hyperref}
\usepackage{amssymb}
\usepackage[T1]{fontenc}
\usepackage{amsthm}

\usepackage{hyperref}
\hypersetup{colorlinks,
	linkcolor={black},
	filecolor={black},
	citecolor={blue},
	urlcolor={blue}}  

\usepackage{amsthm}
\newtheorem{theorem}{Theorem}[section]

\newtheorem{conjecture}{Conjecture}

\newtheorem{corollary}{Corollary}[section]
\newtheorem{lma}{Lemma}[section]
\numberwithin{equation}{section}

\begin{document}
	\title[Arithmetic properties of the 2-color overpartition function $\overline{p}_\ell(n)$]{Arithmetic properties of the 2-color overpartition function $\overline{p}_\ell(n)$}
	\author{H. S. Sumanth Bharadwaj, N. Sujatha and S. Chandankumar}
	
	\address[H. S. Sumanth Bharadwaj]{Department of Mathematics and Statistics, Faculty of Natural Sciences, M. S. Ramaiah University of Applied Sciences, Peenya Campus, Peenya 4th Phase, Bengaluru-560 058, Karnataka, India.} 
	\email{sumanthbharadwaj@gmail.com}
	
	\address[N. Sujatha]{Department of Mathematics, B. M. S. College of Engineering, Bengaluru-560 004, Karnataka, India.} 
	\email{sujathan.maths@bmsce.ac.in}
	
	\address[S. Chandankumar]{Department of Mathematics and Statistics, Faculty of Natural Sciences, M. S. Ramaiah University of Applied Sciences, Peenya Campus, Peenya 4th Phase, Bengaluru-560 058, Karnataka, India.} 
	\email{chandan.s17@gmail.com}
	
	
	\maketitle
	\begin{abstract}
		We derive general families of Ramanujan-type congruences for the function $\overline{p}_{\ell}(n)$, which counts the $2$-color overpartitions of $n$ in which one of the colors appears only in parts that are multiples of $\ell$. For example, for all $n\geq0$ we prove that $$\overline{p}_{4}(32n+28)\equiv 0 \pmod{512}.$$
	\end{abstract}
	
	\section{Introduction}
	
	An \emph{overpartition} of a positive integer $n$ is a partition of $n$
	in which the first occurrence of a part may be overlined.
	Let $\overline{p}(n)$ denote the number of overpartitions of $n$.
	For example, $\overline{p}(3)=8.$ Corteel and Lovejoy \cite{CorteelLovejoy} showed that the generating
	function for $\overline{p}(n)$ is
	\[
	\sum_{n=0}^{\infty}\overline{p}(n)q^n
	=
	\frac{f_2}{f_1^2},
	\]
	where throughout this paper
	\[
	(a;q)_\infty
	=
	\prod_{m=0}^{\infty}(1-aq^m),
	\quad |q|<1,\qquad \mathrm{ and } \qquad f_k=(q^k;q^k)_\infty, \quad k\in \mathbb{N}.
	\]
	
	\noindent Let $p_\ell(n)$ denote the number of $2$-color partitions of $n$ in which one of the colors appears only in parts that are multiples of $\ell$. Its generating
	function is given by
	\begin{equation}
		\sum_{n=0}^{\infty} p_\ell(n)q^n
		=
		\frac{1}{(q;q)_\infty (q^\ell;q^\ell)_\infty}=\frac{1}{f_1f_\ell}.
	\end{equation}
	For example, $p_3(3)=4$; the four such $2$-color partitions of $3$ are
	\[
	3_a,\quad 3_b,\quad 2_a+1_a,\quad 1_a+1_a+1_a,
	\]
	where the subscript $a$ marks the unrestricted color and $b$ the color permitted only on multiples of $\ell$ (here $\ell=3$).
	Let $\overline{p}_\ell(n)$ denote the number of $2$-color overpartitions of $n$ in which one of the colors appears only in parts that are multiples of $\ell$. Its generating function is
	\begin{equation}
		\sum_{n=0}^{\infty} \overline{p}_\ell(n)q^n
		=
		\frac{(-q;q)_\infty(-q^\ell;q^\ell)_\infty}
		{(q;q)_\infty(q^\ell;q^\ell)_\infty}=
		\frac{f_2f_{2\ell}}{f_1^2f_\ell^2}.
		\label{pt}
	\end{equation}
	For example, $\overline{p}_3(3)=10$; the ten such $2$-color overpartitions of $3$ are
	\[
	3_a,\;
	\overline{3}_a,\;
	3_b,\;
	\overline{3}_b,\;
	2_a+1_a,\;
	\overline{2}_a+1_a,\;
	2_a+\overline{1}_a,\;
	\overline{2}_a+\overline{1}_a,\;
	1_a+1_a+1_a,\;
	\overline{1}_a+1_a+1_a.
	\]
	Arithmetic congruences modulo small powers of $2$ and $3$ for $\overline{p}_3(n)$ were investigated by M. S. Mahadeva Naika, S. Shivaprasada Nayaka, and C. Shivashankar in \cite{MSMSSNCS}. The general function $\overline{p}_\ell(n)$, and in particular the way its congruences depend on $\ell$, has not otherwise been treated.
	
	The primary objective of this paper is to give such a treatment, systematically investigating the arithmetic properties of $\overline{p}_{\ell}(n)$ for $\ell=2k,3k,4,6,8,$ and $9$. Our results are of four kinds. First, we prove \emph{$\ell$-uniform families}: a single dissection of the generating function \eqref{pt} yields, for \emph{every} $\ell\ge1$, the congruences of Theorem~\ref{t1} and Corollaries~\ref{c1}--\ref{c2}, so that one identity settles infinitely many functions at once. Second, we obtain \emph{internal congruences} relating a function to a dilate of itself, such as $\overline{p}_{3\ell}(3n+1)\equiv\overline{p}_{3\ell}\!\left(4^\alpha(3n+1)\right)\pmod 4$. Third, for the individual values $\ell=4,6,8,9$ we establish prime-power congruences; the strongest of these is the modulus-$512$ family
	\[
	\overline{p}_{4}(32n+28)\equiv0\pmod{512}\qquad(n\ge0),
	\]
	together with a family of related congruences modulo $128$ and $256$. Fourth, combining the support of $f_1$ and $f_1^3$ on pentagonal and triangular numbers with a quadratic-residue argument modulo an auxiliary prime $p$, we obtain two new vanishing congruences,
	\[
	\overline{p}_{3\ell}(24pn+24b+1)\equiv0\pmod4,\qquad
	\overline{p}_{4}(16pn+16b+2)\equiv0\pmod8,
	\]
	valid for \emph{every} prime $p\ge5$ and every $b$ for which $24b+1$, respectively $8b+1$, is a quadratic non-residue modulo $p$; since infinitely many such $(p,b)$ exist, this yields infinitely many further congruences at a single stroke.
	
	The methods employed throughout the paper are elementary and rely chiefly on classical $q$-series identities, dissections, and generating function manipulations; no appeal is made to the theory of modular forms. Section 2 collects the dissection lemmas used throughout. Section 3 proves the $\ell$-uniform families and their corollaries, together with the internal congruences of Theorem~\ref{t1}. Section 4 treats the individual cases $\ell=4,6,8,9$ in turn, including the two quadratic-residue vanishing congruences described above. Section 5 records further observations and conjectures.
	
	\section{Preliminaries}
	\noindent We begin this section by introducing Ramanujan's general theta-function $f(a,b)$, defined as:
	\begin{align}
		f(a,b): = \sum_{n=-\infty}^{\infty} a^{\frac{n(n+1)}{2}} b^{\frac{n(n-1)}{2}}, \quad |ab| < 1.
	\end{align}
	The following definitions of theta-functions $\varphi$, $\psi$ and $f$ are classical:
	\begin{eqnarray}
		\varphi(q)&:=&f(q,q)= \frac{f_2^5}{f_1^2f_4^2},\\
		\psi(q)&:=&f(q,q^3) =\frac{f_2^2}{f_1},\\
		f(-q)&:=&f(-q,-q^2)=f_1.
	\end{eqnarray}
	\begin{lma}
		Suppose $l \geq 1$, $m \geq 1$ are any integer, and $p$ is any prime. Then we have
		\begin{align}\label{bt}
			f^{p^{l-1}}_{pm} \equiv f_{m}^{p^{l}} \pmod {p^l}.
		\end{align}
	\end{lma}
	\begin{lma}\cite{NBK}\label{l2}
		The following $2$-dissection hold
		\begin{align}
			\dfrac{1}{f_1f_3}&=\dfrac{f_8^2f_{12}^5}{f_2^2f_4f_6^4f_{24}^2}+q\dfrac{f_4^5f_{24}^2}{f_2^4f_6^2f_8^2f_{12}}.\label{1byf1f3}
		\end{align}
	\end{lma}
	\begin{lma}[{\cite{xia2013analogues}}]\label{l2b}
		The following $2$-dissection holds modulo $3$:
		\begin{align}
			\frac{f_1^2\,f_6^4}{f_2^2\,f_3^2}
			\equiv
			\frac{f_4^2\,f_{12}^4}{f_2\,f_6\,f_8\,f_{24}}
			+q\,\frac{f_8\,f_{12}\,f_{24}}{f_4}
			\pmod{3}.\label{f12byf33}
		\end{align}
	\end{lma}
	\begin{lma} [{\cite[Entry~25]{berndt2012ramanujan}}] The following 2-dissections hold: 
		\begin{equation}\label{f12}
			f_1^2 =\frac{f_2 f_8^5}{f_4^2 f_{16}^2}
			-2q\frac{f_2 f_{16}^2}{f_8},
		\end{equation}
		\begin{equation}\label{f14}
			f_1^4 =\frac{f_4^{10}}{f_2^2 f_{8}^4}
			-4q\frac{f_2^2 f_{8}^4}{f_4^2},
		\end{equation}
		\begin{equation}\label{1byf12}
			\frac{1}{f_1^2} =\frac{f_8^5}{f_2^5 f_{16}^2}
			+2q\frac{f_4^2 f_{16}^2}{f_2^5 f_8},
		\end{equation}
		\begin{equation}\label{1byf14}
			\frac{1}{f_1^4} =\frac{f_4^{14}}{f_2^{14} f_{8}^4}
			+4q\frac{f_4^2 f_{8}^4}{f_2^{10}}.
		\end{equation}
	\end{lma}
	\begin{lma}[{\cite{hirschhorn1993cubic}}]The following 2-dissection hold:
		\begin{align}
			\frac{f_3^3}{f_1} &= \frac{f_4^3f_6^2}{f_2^2f_{12}}+q\frac{f_{12}^3}{f_4}.\label{f33byf1}
		\end{align}
	\end{lma}
	
	\begin{lma}[{\cite{mdb}}] The following 3-dissection hold:
		\begin{align}
			\frac{f_4}{f_1}=\frac{f_{12}f_{18}^{4}}{f_3^{3}f_{36}^{2}}+q\,\frac{f_6^{2}f_9^{3}f_{36}}{f_3^{4}f_{18}^{2}}+2q^{2}\,\frac{f_6f_{18}f_{36}}{f_3^{3}}.\label{f4byf1}
	\end{align}\end{lma}
	\begin{lma}[{\cite{mdb}}] The following 3-dissections hold:
		\begin{align}
			&\frac{f_2}{f_1^2} = \frac{f_6^4 f_9^6}{f_3^8 f_{18}^3} + 2q \frac{f_6^3 f_9^3}{f_3^7} + 4q^2 \frac{f_6^2 f_{18}^3}{f_3^6},\label{f2byf12}\\
			&\frac{f_1^2}{f_2} = \frac{f_9^2}{f_{18}} - 2q \frac{f_3 f_{18}^2}{f_6 f_9},\label{f12byf2}\\
			&\frac{f_2^2}{f_1} = \frac{f_6f_9^2}{f_3f_{18}} +q \frac{ f_{18}^2}{ f_9}.\label{f22byf1}
	\end{align}\end{lma}
	\begin{lma} [{\cite{hirschhorn2014congruence}}] The following 3-dissection hold:
		
		\begin{align}
			f_1f_2=\frac{f_6\,f_9^4}{f_3f_{18}^2}
			-qf_9f_{18}-2q^2\frac{f_3\,f_{18}^4}{f_6\,f_9^2}.\label{f1f2}
		\end{align}
	\end{lma}
	
	\begin{lma} [{\cite{xia2013analogues}}] The following 2-dissection hold:
		\begin{align}      
			\frac{f_3^4}{f_1^4}
			&=\frac{f_4^8 f_6^2 f_{12}^{4}}{f_2^{10} f_8^2 f_{24}^{2}}
			+ 4q \frac{f_4^5 f_6^3 f_{12}}{f_2^9}
			+ 4q^2 \frac{f_4^2 f_6^4 f_8^2 f_{24}^{2}}{f_2^8 f_{12}^{2}}.\label{f34f14}
		\end{align}
	\end{lma}
	\begin{lma}[{\cite[Theorem ~2.2]{cui2013}}]\label{a1}
		For any prime $p\ge5$,
		\begin{equation}\label{a17}
			f_1=\sum\limits_{\substack{k=\frac{1-p}{2}\\k\neq\frac{\pm p-1}{6}
			}}^{\frac{p-1}{2}}{(-1)^kq^{\frac{3k^2+k}{2}}f\left(-q^{\frac{3p^2+(6k+1)p}{2}},-q^{\frac{3p^2-(6k+1)p}{2}}\right)}+(-1)^{\frac{\pm p-1}{6}}q^{\frac{p^2-1}{24}}f_{p^2},
		\end{equation}
		where
		\begin{equation}
			\dfrac{\pm p-1}{6}:=\begin{cases}
				\frac{p-1}{6}, & \text{if $p \equiv 1 \pmod{6}$},\\
				\frac{-p-1}{6}, & \text{if $p \equiv -1 \pmod{6}$}.\nonumber
			\end{cases}
		\end{equation}
	\end{lma}
	
	\begin{lma}[{\cite[Theorem ~2.1]{cui2013}}]
		For any odd prime p,
		\begin{align}
			\psi(q)=\sum_{m=0}^{\frac{p-3}{2}}q^{\frac{m^{2}+m}{2}}f\left ( q^\frac{p^2+(2m+1)p}{2},q^\frac{p^2-(2m+1)p}{2} \right )+q^{\frac{p^{2}-1}{8}}\frac{f_{2p^{2}}^{2}}{f_{p^{2}}}.\label{a2}
		\end{align}
	\end{lma}
	Also,
	\begin{equation*}
		\frac{m^{2}+m}{2}\not \equiv \frac{p^{2}-1}{8} \pmod{p} {\hspace{1mm}} \text{for} {\hspace{1mm}} 0\leq m\leq \frac{p-3}{2}.
	\end{equation*}
	
	\begin{lma}[{\cite[Theorem 2.1]{LW}}]
		For any odd prime $p$,
		\begin{equation}
			f_{1}^{3}=\sum_{k=-\frac{p-1}{2}}^{\frac{p-3}{2}}(-1)^kq^{\frac{k^{2}+k}{2}}B_k(q^p)+(-1)^{\frac{p-1}{2}}pq^{\frac{p^{2}-1}{8}}f_{p^{2}}^{3},\label{abc1}
		\end{equation}
	\end{lma}
	where $$B_k(q):=\frac{1}{2}\sum_{n=-\infty}^{\infty}(-1)^n(2pn+2k+1)q^{\dfrac{pn^2+(2k+1)n}{2}}.$$
	
	\section{Congruences for $\overline{p}_{2\ell}(n)$ and $\overline{p}_{3\ell}(n)$}
	In the following theorem, we derive the 2-dissections and 3-dissections of the function $\overline{p}_{2\ell}(n)$ and $\overline{p}_{3\ell}(n)$ respectively. These dissections play a pivotal role in establishing the arithmetic properties of the partition function $\overline{p}_{\ell}(n)$.
	\begin{theorem}\label{t1} For all $n\geq0$ and $\ell\geq1$,
		\begin{align}
			\sum_{n=0}^{\infty}\overline{p}_{2\ell}(2n)q^n=\frac{f_{2\ell}f_4^5}{f_{\ell}^2f_1^4f_{8}^2}=\frac{\varphi(q^2)}{\varphi(-q^\ell)\varphi^2(-q)},\label{res4}\\
			\sum_{n=0}^{\infty}\overline{p}_{2\ell}(2n+1)q^n=2\,\frac{f_{2\ell}f_2^2f_{8}^2}
			{f_{\ell}^2f_1^4f_4} =\frac{\psi(q^4)}{\varphi(-q^\ell)\varphi^2(-q)},\label{res5}\\
			\sum_{n=0}^{\infty}\overline{p}_{3\ell}(3n)q^n=
			\frac{f_{2\ell}f_2^4f_3^6}{f_{\ell}^2f_1^8f_{6}^3}=\frac{\varphi^3(-q^3)}{\varphi(-q^\ell)\varphi^4(-q)},\label{res1}\\
			\sum_{n=0}^{\infty}\overline{p}_{3\ell}(3n+1)q^n=
			2\,\frac{f_{2\ell}f_2^3f_3^3}{f_{\ell}^2f_1^7},\label{res2}\\
			\sum_{n=0}^{\infty}\overline{p}_{3\ell}(3n+2)q^n=
			4\,\frac{f_{2\ell}f_2^2f_{6}^3}{f_{\ell}^2f_1^6}.\label{res3}
		\end{align}
	\end{theorem}
	\begin{proof}
		Invoking \eqref{f2byf12} in \eqref{pt} with $\ell=3\ell$, we have
		\begin{align}\label{r1}
			\sum_{n=0}^{\infty}\overline{p}_{3\ell}(n)q^n=
			\frac{f_{6\ell}f_6^4f_9^6}{f_{3\ell}^2f_3^8f_{18}^3}
			+2q\,\frac{f_{6\ell}f_6^3f_9^3}{f_{3\ell}^2f_3^7}
			+4q^2\,\frac{f_{6\ell}f_6^2f_{18}^3}{f_{3\ell}^2f_3^6}.
		\end{align}
		Equations \eqref{res1}--\eqref{res3} directly follows by extracting the terms involving $q^{3n}$, $q^{3n+1}$ and $q^{3n+2}$ respectively from the equation \eqref{r1}. Now, invoking \eqref{1byf12} in \eqref{pt} with $\ell=2\ell$, we have
		\begin{align}
			\sum_{n=0}^{\infty}\overline{p}_{2\ell}(n)q^n=\frac{f_{4\ell}f_8^5}{f_{2\ell}^2f_2^4f_{16}^2}
			+2q\,\frac{f_{4\ell}f_4^2f_{16}^2}
			{f_{2\ell}^2f_2^4f_8}. \label{r2}
		\end{align}
		Equations \eqref{res4} and \eqref{res5} are obtained  by extracting the terms involving $q^{2n}$ and $q^{2n+1}$ respectively from the equation \eqref{r2}. This completes the proof.
	\end{proof}
	\begin{corollary}\label{c1} Suppose $t$ and $\ell$ are natural numbers and $4|t$ and $9|\ell$, then
		\begin{align}
			\overline{p}_{2t}(4n+2)\equiv 0\pmod{4},\label{cor1}\\
			\overline{p}_{2t}(4n+3)\equiv 0\pmod{8},\label{cor2}\\
			\overline{p}_{3\ell}(9n+3j)\equiv 0\pmod{8}, \,\,\text{for}\,\,j=1,2.\label{cor3}
		\end{align}
	\end{corollary}
	\begin{proof}
		These congruences follow on specializing $t$ and $\ell$ in \eqref{res4}, \eqref{res5}, and \eqref{res1}, reducing modulo the stated modulus by \eqref{bt}, and extracting the relevant arithmetic progression.
	\end{proof}
	\begin{corollary}\label{c2} For all $n\geq0$ and $\ell\geq1$,
		\begin{align}
			\sum_{n=0}^{\infty}\overline{p}_{2\ell}(2n)q^n\equiv \frac{1}{\varphi(-q^2)\varphi(-q^\ell)}\pmod{4},\label{res6}\\
			\sum_{n=0}^{\infty}\overline{p}_{3\ell}(3n)q^n\equiv
			\frac{\varphi^3(-q^3)}{\varphi(-q^\ell)}\pmod{8},\label{res8}\\
			\sum_{n=0}^{\infty}\overline{p}_{2\ell}(2n+1)q^n\equiv2\,\frac{\psi(q^4)}
			{\varphi(-q^\ell)}\pmod{8}, \label{res7}\\
			\overline{p}_{3\ell}(3n+1)\equiv 0 \pmod{2},\\
			\overline{p}_{3\ell}(3n+2)\equiv0 \pmod{4}.
		\end{align}
	\end{corollary}
	\begin{proof}
		Each congruence follows on reducing the dissections of Theorem~\ref{t1} modulo the stated modulus by \eqref{bt}.
	\end{proof}
	\begin{theorem} For all $n\geq0$, $\alpha>1$
		\begin{align}
			\overline{p}_{3\ell}(48n+6j+1)&\equiv0 \pmod{4},\,\, \text{for}\,\,j\in\{1,2,3,5,6,7\},\label{mares1}\\
			\overline{p}_{3\ell}(3n+1)&\equiv\overline{p}_{3\ell}(4^\alpha(3n+1))\pmod{4},\label{mares2}\\
			\overline{p}_{3\ell}(4^{\alpha+1}(6n+5))&\equiv0\pmod{4},\label{mares31}\\
			\overline{p}_{3\ell}(4^\alpha(6n+5))&\equiv0\pmod{8},\label{mares3}\\
			\overline{p}_{3\ell}(3n+2)&\equiv\overline{p}_{3\ell}(4^{\alpha}(3n+2))\pmod{8}.\label{mares4}
		\end{align}
	\end{theorem}
	\begin{proof}
		Thanks to \eqref{bt}, \eqref{res2} takes the form
		\begin{align}\label{t21}
			\sum_{n=0}^{\infty}\overline{p}_{3\ell}(3n+1)q^n\equiv
			2\,\frac{f_{3}^3}{f_1}\pmod{4}.
		\end{align}
		Invoking \eqref{f33byf1} in the above equation and extracting the terms containing $q^{2n}$ and $q^{2n+1}$ from both sides of the resulting equation, we obtain
		\begin{align}\label{t22}
			\sum_{n=0}^{\infty}\overline{p}_{3\ell}(6n+1)q^n\equiv
			2\,f_4\pmod{4}
		\end{align} and 
		\begin{align}\label{t23}
			\sum_{n=0}^{\infty}\overline{p}_{3\ell}(6n+4)q^n\equiv
			2\,\frac{f_{6}^3}{f_2}\pmod{4}.
		\end{align}
		Congruence \eqref{mares1} follows from the above equation \eqref{t22}, since by the pentagonal number theorem $f_4$ is supported only on exponents $4\cdot k(3k-1)/2$, which are congruent to $0$ or $4$ modulo $8$; hence the coefficient of $q^n$ on the right of \eqref{t22} vanishes modulo $4$ whenever $n\not\equiv0,4\pmod8$, that is, for $n=8m+j$ with $j\in\{1,2,3,5,6,7\}$. Equation \eqref{t23} implies that
		\begin{align}\label{t24}
			\sum_{n=0}^{\infty}\overline{p}_{3\ell}(12n+4)q^n\equiv
			2\,\frac{f_{3}^3}{f_1}\pmod{4}.
		\end{align}
		Congruence \eqref{mares2} follows from \eqref{t21} and \eqref{t24}. Thanks to \eqref{bt}, the equation \eqref{res3} takes the form
		\begin{align}\label{t25}
			\sum_{n=0}^{\infty}\overline{p}_{3\ell}(3n+2)q^n\equiv
			4\,\frac{f_{6}^3}{f_2}\pmod{8}
		\end{align}
		which implies
		\begin{align}\label{t26}
			\sum_{n=0}^{\infty}\overline{p}_{3\ell}(6n+2)q^n\equiv
			4\,\frac{f_3^3}{f_{1}}\pmod{8}
		\end{align}
		and
		\begin{align}\label{t27}
			\overline{p}_{3\ell}(6n+5)\equiv0\pmod{8}.
		\end{align}
		Invoking \eqref{bt} in the equation \eqref{t26}, and extracting the terms involving $q^{2n+1}$ from both sides of the resulting equation, we get
		\begin{align}\label{t28}
			\sum_{n=0}^{\infty}\overline{p}_{3\ell}(12n+8)q^n\equiv
			4\,\frac{f_{6}^3}{f_2}\pmod{8}.
		\end{align}
		Congruence \eqref{mares4} follows from \eqref{t25} and above equation \eqref{t28}. This completes the proof.
	\end{proof}
	
	The dissection of $f_1$ recorded in Lemma~\ref{a1} restricts the support of $f_1$
	to a single residue class modulo $p$ in each of its $p$ pieces, since for
	$k\neq\frac{\pm p-1}{6}$ the two theta-arguments
	$\frac{3p^2+(6k+1)p}{2}$ and $\frac{3p^2-(6k+1)p}{2}$ are themselves multiples of
	$p$; the leading exponent of the $k$-th piece is therefore the residue class it
	occupies. The same fact follows directly from Euler's pentagonal number theorem:
	$f_1$ is supported exactly on the pentagonal numbers $k(3k-1)/2$, $k\in\mathbb Z$,
	and $24m+1$ is a perfect square whenever $m$ is pentagonal. Consequently, for a
	prime $p\ge5$ and $0\le b<p$, if $24b+1$ is a quadratic non-residue modulo $p$
	then no integer $m\equiv b\pmod p$ is pentagonal, and the coefficient of $q^m$ in
	$f_1$ vanishes identically for every such $m$.
	
	\begin{theorem}\label{thpdiss1}
		Let $p\ge5$ be prime and let $b$, $0\le b<p$, satisfy
		$\left(\dfrac{24b+1}{p}\right)=-1$. Then for every $n\ge0$,
		\begin{equation}
			\overline{p}_{3\ell}(24pn+24b+1)\equiv0\pmod4.
			\label{pdiss3l}
		\end{equation}
	\end{theorem}
	\begin{proof}
		Since $f_4=f_1(q^4)$, the coefficient of $q^m$ in $f_4$ vanishes unless $4\mid m$,
		in which case it equals the coefficient of $q^{m/4}$ in $f_1$. By the discussion
		above, this coefficient is identically zero whenever $m/4\equiv b\pmod p$ with
		$24b+1$ a quadratic non-residue mod $p$, that is, whenever
		$m\equiv4b\pmod{4p}$. Equation \eqref{t22} gives
		$\overline{p}_{3\ell}(6n+1)\equiv2f_4\pmod4$, so taking $n=4(pn'+b)$ in $f_4$'s
		expansion forces $\overline{p}_{3\ell}\bigl(6\cdot4(pn'+b)+1\bigr)
		=\overline{p}_{3\ell}(24pn'+24b+1)\equiv0\pmod4$ for every $n'\ge0$.
	\end{proof}
	For example, $p=5$ gives the non-residues $b=3,4$ (since $73\equiv3$ and
	$97\equiv2\pmod5$ are not squares mod $5$), so
	\[
	\overline{p}_{3\ell}(120n+73)\equiv0\pmod4,\qquad
	\overline{p}_{3\ell}(120n+97)\equiv0\pmod4
	\]
	for every $\ell\ge1$ and every $n\ge0$.
	
	\section{Congruences for $\overline{p}_4(n)$, $\overline{p}_6(n)$, $\overline{p}_8(n)$, and $\overline{p}_9(n)$}
	\subsection{Congruences for $\overline{p}_{4}(n)$}
	In this section, we study the arithmetic properties of $\overline{p}_{4}(n)$.
	\begin{theorem}For all $n\geq0$, $\alpha\geq1$
		\begin{align}
			\overline{p}_{4}(2n)&\equiv\overline{p}_{4}(2^{\alpha}(2n))  \pmod{4},\\
			\overline{p}_{4}(16n+10)&\equiv 0 \pmod{8},\\
			\overline{p}_{4}(4^{\alpha}(8n+6))&\equiv 0 \pmod{16},\\
			\overline{p}_{4}(16n+14)&\equiv 0 \pmod{32},\\
			\overline{p}_{4}(2^{\alpha}(16n+12))&\equiv 0 \pmod{32},\\
			\overline{p}_{4}(16n+12)&\equiv 0 \pmod{256},\\
			\overline{p}_{4}(32n+28)&\equiv 0 \pmod{512}.
		\end{align}
	\end{theorem}
	\begin{proof}
		Substituting $\ell=2$ into \eqref{res4} and \eqref{res5}, and invoking \eqref{1byf14} in the resulting equations, we get
		\begin{align}
			\sum_{n=0}^{\infty}\overline{p}_{4}(4n)q^n=\frac{f_2^{20}}{f_1^{16}f_4^6},\label{p41}\\
			\sum_{n=0}^{\infty}\overline{p}_{4}(4n+1)q^n= 2\frac{f_2^{14}}{f_1^{14}f_4^2},\label{p50}\\
			\sum_{n=0}^{\infty}\overline{p}_{4}(4n+2)q^n=4\frac{f_2^{8}f_4^2}{f_1^{12}},\label{p42}\\
			\sum_{n=0}^{\infty}\overline{p}_{4}(4n+3)q^n=8\,\frac{f_2^2\,f_4^{6}}{f_1^{10}}.\label{p510}
		\end{align}
		Note that equations \eqref{p41}--\eqref{p510} provide the components of the 4-dissection of the generating function for $\overline{p}_{4}(n)$. Now, invoking \eqref{1byf14} in \eqref{p41} and extracting the terms containing $q^{2n}$ and $q^{2n+1}$ from both sides of the resulting equation, we get
		\begin{align}
			\sum_{n=0}^{\infty}\overline{p}_{4}(8n)q^n=\frac{f_2^{50}}{f_1^{36}f_{4}^{16}} +  96q\frac{f_2^{26}}{f_1^{28}}+256q^2\frac{f_2^{2} f_{4}^{16}}{f_1^{20}}\label{p43}
		\end{align}
		and
		\begin{align}
			\sum_{n=0}^{\infty}\overline{p}_{4}(8n+4)q^n=16\frac{f_2^{38}}{f_1^{32} f_{4}^{8}} +  256q\frac{f_2^{14} f_{4}^{8}}{f_1^{24}}.\label{p44}
		\end{align}
		Thanks to \eqref{bt}, the equation \eqref{p43} reduces to
		\begin{align}
			\sum_{n=0}^{\infty}\overline{p}_{4}(8n)q^n\equiv\frac{f_2^{34}}{f_1^{4}f_{4}^{16}}\pmod{32}.\label{p43i1}
		\end{align}
		Invoking \eqref{1byf14} in the above equation and extracting the terms involving $q^{2n}$ and $q^{2n+1}$ from both sides of the resulting equation, we get
		\begin{align}
			\sum_{n=0}^{\infty}\overline{p}_{4}(16n)q^n\equiv\frac{f_1^{20}}{f_2^{2}f_{4}^{4}}\pmod{32}\label{p43i2}
		\end{align}
		and
		\begin{align}
			\sum_{n=0}^{\infty}\overline{p}_{4}(16n+8)q^n\equiv4\frac{f_1^{24}f_{4}^{4}}{f_2^{14}}\equiv4\frac{f_{4}^{4}}{f_2^{2}}\pmod{32}.\label{p43i4}
		\end{align}
		Extracting the terms involving $q^{2n}$ and $q^{2n+1}$ respectively from the equation \eqref{p43i4}, we get
		\begin{align}
			\sum_{n=0}^{\infty}\overline{p}_{4}(32n+8)q^n\equiv4\frac{f_{2}^{4}}{f_1^{2}}\pmod{32}\label{p43i5}
		\end{align}
		and 
		\begin{align}
			\overline{p}_{4}(32n+24)\equiv0\pmod{32}.\label{p43i510}
		\end{align}
		Invoking \eqref{1byf12} in the equation \eqref{p43i5} and extracting the terms containing $q^{2n+1}$ in the resulting equation, we have
		\begin{align}
			\sum_{n=0}^{\infty}\overline{p}_{4}(64n+40)q^n\equiv8\psi(q)\psi(q^4)\pmod{32}.\label{p43i6}
		\end{align}
		Substituting \eqref{f14} in the equation \eqref{p43i2} and extracting the terms involving $q^{2n+1}$ from both sides of the resulting equation, we get
		\begin{align}
			\sum_{n=0}^{\infty}\overline{p}_{4}(32n+16)q^n\equiv12\frac{f_2^{34}}{f_1^{8}f_{4}^{12}}\equiv12\frac{f_2^{30}}{f_{4}^{12}} \pmod{32}\label{p43i3}
		\end{align}
		which implies 
		\begin{align}
			\overline{p}_{4}(64n+48)\equiv0 \pmod{32}.\label{p43i31}
		\end{align}
		Thanks to \eqref{bt}, the equation \eqref{p44} takes the form
		\begin{align}
			\sum_{n=0}^{\infty}\overline{p}_{4}(8n+4)q^n\equiv16\frac{f_2^{22}}{f_{4}^{8}}+256qf_2^2f_4^8\pmod{512}\label{p45}
		\end{align}
		which implies
		\begin{align}
			\sum_{n=0}^{\infty}\overline{p}_{4}(16n+4)q^n\equiv16\frac{f_1^{22}}{f_{2}^{8}}\equiv16\frac{f_2^4}{f_1^2}\pmod{128}\label{p46}
		\end{align}
		and
		\begin{align}
			\sum_{n=0}^{\infty}\overline{p}_{4}(16n+12)q^n\equiv256f_1^2f_2^8\equiv 256f_2f_4^4\pmod{512}.\label{p47}
		\end{align}
		Invoking \eqref{1byf12} in the equation \eqref{p46} and extracting the terms involving $q^{2n}$ and $q^{2n+1}$ from the resulting equation, we get
		\begin{align}
			\sum_{n=0}^{\infty}\overline{p}_{4}(32n+4)q^n\equiv16\frac{f_4^{5}}{f_{1}f_{8}^2}\equiv16\varphi(q^2)\psi(q)\pmod{128}\label{p48}
		\end{align}
		and
		\begin{align}
			\sum_{n=0}^{\infty}\overline{p}_{4}(32n+20)q^n\equiv32\psi(q)\psi(q^4)\pmod{128}.\label{p49}
		\end{align}
		Invoking \eqref{1byf14} in \eqref{p42} and extracting the terms containing $q^{2n}$ and $q^{2n+1}$ from both sides of the resulting equation, we get
		\begin{align}
			\sum_{n=0}^{\infty}\overline{p}_{4}(8n+2)q^n=4\frac{f_2^{44}}{f_1^{34}f_4^{12}}+
			192q\,\frac{f_2^{20}f_4^{4}}{f_1^{26}}\label{p431}
		\end{align}
		and
		\begin{align}
			\sum_{n=0}^{\infty}\overline{p}_{4}(8n+6)q^n=48\,\frac{f_2^{32}}{f_1^{30}f_4^{4}}+
			256q\,\frac{f_2^{8}f_4^{12}}{f_1^{22}}.\label{p432}
		\end{align}
		Thanks to \eqref{bt}, \eqref{p431} and \eqref{p432} respectively takes the form
		\begin{align}
			\sum_{n=0}^{\infty}\overline{p}_{4}(8n+2)q^n\equiv4f_2^{3}\pmod{8}\label{p433}
		\end{align}
		and 
		\begin{align}
			\sum_{n=0}^{\infty}\overline{p}_{4}(8n+6)q^n\equiv16f_2^{9}\pmod{32}\label{p434}
		\end{align}
		the above equation \eqref{p434} implies 
		\begin{align}
			\overline{p}_{4}(16n+14)\equiv0\pmod{32}\label{p435}
		\end{align}
		and
		\begin{align}
			\sum_{n=0}^{\infty}\overline{p}_{4}(16n+6)q^n\equiv 16f_1^{9}\equiv 16f_1f_8\pmod{32}.\label{p436}
		\end{align}
		
	\end{proof}
	
	The $f_1^3$ dissection of the lemma preceding \eqref{abc1} exhibits the same
	structure as Lemma~\ref{a1}: for each $k$, the inner sum defining $B_k(q^p)$ has
	every exponent $\frac{pn^2+(2k+1)n}{2}$ a multiple of $p$ once shifted by the
	prefactor $q^{(k^2+k)/2}$, so the $k$-th piece is again confined to a single
	residue class modulo $p$. This is equally transparent from Jacobi's identity
	$f_1^3=\sum_{m\ge0}(-1)^m(2m+1)q^{m(m+1)/2}$: the support of $f_1^3$ is exactly
	the triangular numbers $m(m+1)/2$, $m\ge0$, and $8\cdot\frac{m(m+1)}{2}+1=(2m+1)^2$
	is always a perfect square. Hence, for a prime $p\ge5$ and $0\le b<p$, if $8b+1$
	is a quadratic non-residue modulo $p$ then no integer $r\equiv b\pmod p$ is
	triangular, and the coefficient of $q^r$ in $f_1^3$ vanishes identically for
	every such $r$.
	
	\begin{theorem}\label{thpdiss2}
		Let $p\ge5$ be prime and let $b$, $0\le b<p$, satisfy
		$\left(\dfrac{8b+1}{p}\right)=-1$. Then for every $n\ge0$,
		\begin{equation}
			\overline{p}_{4}(16pn+16b+2)\equiv0\pmod8.
			\label{pdiss4}
		\end{equation}
	\end{theorem}
	\begin{proof}
		Since $f_2^3=f_1^3(q^2)$, the coefficient of $q^m$ in $f_2^3$ vanishes unless
		$2\mid m$, in which case it equals the coefficient of $q^{m/2}$ in $f_1^3$. By
		the discussion above, this coefficient is identically zero whenever
		$m/2\equiv b\pmod p$, that is, whenever $m\equiv2b\pmod{2p}$. Equation
		\eqref{p433} gives $\overline{p}_4(8n+2)\equiv4f_2^3\pmod8$, so taking
		$n=2(pn'+b)$ in $f_2^3$'s expansion forces
		$\overline{p}_4\bigl(8\cdot2(pn'+b)+2\bigr)=\overline{p}_4(16pn'+16b+2)\equiv0
		\pmod8$ for every $n'\ge0$.
	\end{proof}
	For example, $p=5$ gives the non-residues $b=2,4$ for the hypothesis
	$\left(\frac{8b+1}{p}\right)=-1$ (the squares modulo $5$ are $\{0,1,4\}$,
	while $17\equiv2$ and $33\equiv3\pmod5$ are not among them); substituting
	these into the conclusion $16pn+16b+2$ with $p=5$ gives
	\[
	\overline{p}_4(80n+34)\equiv0\pmod8,\qquad
	\overline{p}_4(80n+66)\equiv0\pmod8
	\]
	for every $n\ge0$.
	
	\subsection{Congruences for $\overline{p}_{6}(n)$}
	In this section, we establish a few congruences for the partition function  $\overline{p}_{6}(n)$.
	\begin{theorem}\label{th6}For all $n\geq0$, $\alpha\geq1$
		\begin{align}
			\overline{p}_{6}(9n+6)\equiv0\pmod{3},\label{res61}\\
			\overline{p}_{6}(8n+3)\equiv0\pmod{8},\label{res62}\\
			\overline{p}_{6}(8n+5)\equiv0\pmod{8},\label{res63}\\
			\overline{p}_{6}(24n+17)\equiv0\pmod{8},\label{res64}\\
			\overline{p}_{6}(24n+23)\equiv0\pmod{8}.\label{res65}
		\end{align}
	\end{theorem}
	\begin{proof}
		Substituting $\ell=3$ in the equation \eqref{res5}, we obtain
		\begin{align}\label{p63}
			\sum_{n=0}^{\infty}\overline{p}_{6}(2n+1)q^n\equiv2\frac{f_2^2f_6f_8^2}{f_1^4f_3^2f_4}\equiv2\frac{f_6f_8^2}{f_3^2f_4}\pmod{8}.
		\end{align}
		Invoking \eqref{1byf12} in the above equation \eqref{p63}, we get
		\begin{align}
			\sum_{n=0}^{\infty}\overline{p}_{6}(4n+1)q^n\equiv2
			\frac{f_4^2f_{12}^5}
			{f_2f_3^4f_{24}^2}\equiv2\frac{f_4^2f_{12}}{f_2f_6^2}\pmod{8}\label{p64}
		\end{align}
		and
		\begin{align}
			\sum_{n=0}^{\infty}\overline{p}_{6}(4n+3)q^n\equiv4q\frac{f_4^2f_{6}^2f_{24}^2}
			{f_2f_3^4f_{12}}\pmod{8}.\label{p65}
		\end{align}
		Congruences \eqref{res62} and \eqref{res63} follow immediately from \eqref{p65} and \eqref{p64}, respectively. Since the right-hand side of \eqref{p64} is a power series in $q^2$, replacing $q^2$ by $q$ gives
		\begin{align}
			\sum_{n=0}^{\infty}\overline{p}_{6}(8n+1)q^n\equiv2\frac{f_2^2f_6}{f_1f_3^2}\pmod{8}.\label{p66}
		\end{align}
		By extracting the coefficients of $q^{3n+2}$ from both sides of \eqref{p66}, we obtain congruence \eqref{res64}. Now by extracting the terms involving $q^{2n+1}$ from both sides of the equation \eqref{p65}, we get
		\begin{align}
			\sum_{n=0}^{\infty}\overline{p}_{6}(8n+7)q^n\equiv4\frac{f_2^2f_{12}^2}
			{f_1f_{6}}\equiv4\frac{f_9^2f_{12}^2}{f_3f_{18}}
			+4q\,\frac{f_{12}^2f_{18}^2}{f_6f_9}
			\pmod{8}.\label{p67}
		\end{align}
		Extracting the terms involving $q^{3n+2}$ from both sides of the above equation, we arrive at the congruence \eqref{res65}.
		Substituting $\ell=2$ in \eqref{res1}, we get
		\begin{align}\label{p61}
			\sum_{n=0}^{\infty}\overline{p}_{6}(3n)q^n\equiv\frac{f_2^2f_3^6f_4}{f_1^8f_6^3}\equiv\frac{f_4}{f_1}\frac{f_2^2}{f_1}\frac{f_3^4}{f_6^3}\pmod{3}.
		\end{align}
		Substituting \eqref{f4byf1} and \eqref{f22byf1} in the equation \eqref{p61}, we get
		\begin{align}\label{p62}
			\sum_{n=0}^{\infty}\overline{p}_{6}(3n)q^n\equiv
			\frac{f_9^2f_{12}f_{18}^3}{f_6^2f_{36}^2}+q\left(
			\frac{f_3f_{12}f_{18}^6}{f_6^3f_9f_{36}^2}+\frac{f_9^5f_{36}}{f_3f_{18}^3}\right)+3q^2\frac{f_9^2f_{36}}{f_6}+2q^3\frac{f_3f_{18}^3f_{36}}{f_6^2f_9}.
			\pmod{3}.
		\end{align}
		By extracting the coefficients of $q^{3n+2}$ from both sides of \eqref{p62}, we obtain congruence \eqref{res61}. This completes the proof.
	\end{proof}
	\subsection{Congruences for $\overline{p}_{8}(n)$}
	In this section, we establish a few congruences for the partition function $\overline{p}_{8}(n)$.
	\begin{theorem}\label{61}For all $n\geq0$, $\alpha\geq1$
		\begin{align}
			\overline{p}_{8}(6n+5)\equiv0\pmod{3}.\label{res81}\\
			\overline{p}_{8}(16n+10)\equiv0\pmod{16},\label{res82}\\
			\overline{p}_{8}(2^{\alpha}(16n+10))\equiv0\pmod{16},\label{res83}\\
			\overline{p}_{8}(32n+28)\equiv0\pmod{16},\label{res84}\\
			\overline{p}_{8}(16n+14)\equiv0\pmod{32},\label{res85}\\
			\overline{p}_{8}(2^\alpha(8n+7))\equiv0\pmod{32},\label{res86}\\
			\overline{p}_{8}(8n+7)\equiv0\pmod{64}.\label{res87}
		\end{align}
	\end{theorem}
	\begin{proof}
		Substituting $\ell=4$ in \eqref{res5} and owing to \eqref{bt}, we have
		\begin{align}
			\sum_{n=0}^{\infty}\overline{p}_{8}(2n+1)q^n\equiv2\,\frac{f_2^2f_8^3}
			{f_1^4f_4^3}\equiv\frac{f_2^2f_{24}}
			{f_1f_3f_{12}}\pmod{3}.
		\end{align}
		Invoking \eqref{f22byf1} in the above equation and extracting the terms containing $q^{3n+2}$ from both sides of the resulting equation, we arrive at the congruence \eqref{res81}.
		Now Substituting $\ell=4$ in \eqref{res4} and \eqref{res5} and using \eqref{1byf14}, we get 
		\begin{align}
			\sum_{n=0}^{\infty}\overline{p}_{8}(4n)q^n=\frac{f_2^{17}}{f_1^{14}f_4^5},\label{p81}\\
			\sum_{n=0}^{\infty}\overline{p}_{8}(4n+1)q^n=2\,\frac{f_2^{11}}{f_1^{12}f_4},\label{p4321}\\
			\sum_{n=0}^{\infty}\overline{p}_{8}(4n+2)q^n=4\,\frac{f_2^5f_4^3}{f_1^{10}},\label{p82}\\
			\sum_{n=0}^{\infty}\overline{p}_{8}(4n+3)q^n=8\,\frac{f_4^7}{f_1^8f_2}.\label{p4322}
		\end{align}
		Thanks to \eqref{bt}, equation \eqref{p81} takes the form
		\begin{align}
			\sum_{n=0}^{\infty}\overline{p}_{8}(4n)q^n\equiv\frac{f_1^2f_2^{9}}{f_4^5}\pmod{16}.\label{p811}
		\end{align}
		Invoking \eqref{f12} in the \eqref{p811} and extracting the terms involving $q^{2n+1}$ from both sides of the resulting equation, we get 
		\begin{align}
			\sum_{n=0}^{\infty}\overline{p}_{8}(8n+4)q^n\equiv-2\frac{f_1^{10}f_8^{2}}{f_2^5f_4}\equiv-2\frac{f_1^{2}f_8^{2}}{f_2f_4}\pmod{16}.\label{p812}
		\end{align}
		Invoking \eqref{f12} and extracting the terms involving $q^{2n}$  and $q^{2n+1}$ from both sides of the resulting equation, we get
		\begin{align}
			\sum_{n=0}^{\infty}\overline{p}_{8}(16n+4)q^n\equiv-2\frac{f_4^7}{f_2^3f_8^2}\pmod{16}.\label{p8131}
		\end{align} 
		and
		\begin{align}
			\sum_{n=0}^{\infty}\overline{p}_{8}(16n+12)q^n\equiv4\frac{f_4f_8^{2}}{f_2}\pmod{16}.\label{p813}
		\end{align}    
		Extracting the terms involving $q^{2n}$ from both sides of the above equation, we get
		\begin{align}
			\sum_{n=0}^{\infty}\overline{p}_{8}(32n+12)q^n\equiv4\frac{f_2f_4^{2}}{f_1}\equiv 4 \psi(q)f_2^3\pmod{16}.\label{p814}
		\end{align}
		Invoking \eqref{1byf14} in \eqref{p4322} and extracting the terms involving from both sides of the resulting equation, we get
		\begin{align}
			\sum_{n=0}^{\infty}\overline{p}_{8}(8n+3)q^n=8\,\frac{f_2^{35}}{f_1^{29}f_4^8}+
			128q\,\frac{f_2^{11}f_4^8}{f_1^{21}}.\label{p4323}
		\end{align} and
		\begin{align}
			\sum_{n=0}^{\infty}\overline{p}_{8}(8n+7)q^n=
			64\,\frac{f_2^{23}}{f_1^{25}}.\label{p4324}
		\end{align}
		Thanks to \eqref{bt}, equation \eqref{p4323} reduces to 
		\begin{align}
			\sum_{n=0}^{\infty}\overline{p}_{8}(8n+3)q^n\equiv8f_2^3 \psi(q)\pmod{32}.\label{p43231}
		\end{align}
		Invoking \eqref{1byf14} in \eqref{p4321} and extracting the terms involving from both sides of the resulting equation, we get
		\begin{align}
			\sum_{n=0}^{\infty}\overline{p}_{8}(8n+1)q^n=2\,\frac{f_2^{41}}{f_1^{31}f_4^{12}}+96q\,\frac{f_2^{17}f_4^4}{f_1^{23}}\label{p4325}
		\end{align}
		and 
		\begin{align}
			\sum_{n=0}^{\infty}\overline{p}_{8}(8n+5)q^n=
			24\,\frac{f_2^{29}}{f_1^{27}f_4^4}+ 128q\,\frac{f_2^{5}f_4^{12}}{f_1^{19}}.\label{p4326}
		\end{align}
		Thanks to \eqref{bt}, equation \eqref{p4325} and  \eqref{p4326} respectively reduces to
		\begin{align}
			\sum_{n=0}^{\infty}\overline{p}_{8}(8n+1)q^n\equiv2f_1f_2\pmod{16}\label{p43251}
		\end{align}
		and
		\begin{align}
			\sum_{n=0}^{\infty}\overline{p}_{8}(8n+5)q^n\equiv24\psi(q)f_4^3\pmod{16}\label{p43252}.
		\end{align}
		Invoking \eqref{1byf12} and \eqref{1byf14} in the equation \eqref{p82}, and extracting the terms 
		involving from both sides of the resulting equation, we get
		\begin{align}
			\sum_{n=0}^{\infty}\overline{p}_{8}(8n+2)q^n=4\frac{f_2^{31}}{f_1^{28}f_4^3 f_{8}^2} 
			+ 64q\,\frac{f_2^{21}f_{8}^2}{f_1^{24}f_4} 
			+ 64q\,\frac{f_2^7 f_4^{13}}{f_1^{20}f_{8}^2} 
			\label{p8201}
		\end{align}
		and 
		\begin{align}
			\sum_{n=0}^{\infty}\overline{p}_{8}(8n+6)q^n= 8\,\frac{f_2^{33}f_{8}^2}{f_1^{28}f_4^9} 
			+32\,\frac{f_2^{19}f_4^5}{f_1^{24}f_{8}^2} 
			+ 128q\,\frac{f_2^9 f_4^7 f_{8}^2}{f_1^{20}}\label{p821}.
		\end{align}
		Invoking \eqref{bt} in \eqref{p82}, we get
		\begin{align}
			\sum_{n=0}^{\infty}\overline{p}_{8}(8n+2)q^n\equiv4{f_2f_4} 
			\pmod{16}\label{p822}
		\end{align}
		Thanks to \eqref{bt}, the equation \eqref{p8201} reduces to
		\begin{align}
			\sum_{n=0}^{\infty}\overline{p}_{8}(8n+2)q^n\equiv4\frac{f_1^4f_2^{15}}{f_4^3 f_{8}^2} \pmod{32}.
			\label{p823}
		\end{align}
		Invoking \eqref{f14} in the above equation and extracting the terms involving $q^{2n+1}$ on both sides of the resulting equation, we get
		\begin{align}
			\sum_{n=0}^{\infty}\overline{p}_{8}(16n+10)q^n\equiv16\frac{f_1^{17}f_4^{2}}{f_2^5}\equiv16\psi(q)f_4^3 \pmod{32}.
			\label{p824}
		\end{align}
		Thanks to \eqref{bt}, equation \eqref{p821} takes the form
		\begin{align}
			\sum_{n=0}^{\infty}\overline{p}_{8}(8n+6)q^n\equiv 8\,\frac{f_4^{5}}{f_2}\pmod{32}.\label{p826}
		\end{align}
		Extracting the terms involving $q^{2n}$ on both sides of the above equation, we get
		\begin{align}
			\sum_{n=0}^{\infty}\overline{p}_{8}(16n+6)q^n\equiv 8\,\frac{f_2^{5}}{f_1}\equiv8\psi(q)f_2^3\pmod{32}.\label{p825}
		\end{align}
		Since the right-hand side of \eqref{p826} is a power series in $q^2$, extracting the terms involving $q^{2n+1}$ from both sides yields the congruence \eqref{res85}, which completes the proof.
	\end{proof}
	\begin{theorem}\label{thf1f2}
		For all $n\geq0$,
		\begin{align}
			\overline{p}_{8}(72n+33)&\equiv0\pmod{16},\label{p8new1}\\
			\overline{p}_{8}(72n+57)&\equiv0\pmod{16}.\label{p8new2}
		\end{align}
	\end{theorem}
	\begin{proof}
		Invoking \eqref{f1f2} in \eqref{p43251} and extracting the term involving $q^{3n+1}$ from both sides of the resulting equation,
		we get
		\begin{align}
			\sum_{n=0}^{\infty}\overline{p}_{8}(24n+9)q^n&\equiv-2f_3f_6\pmod{16}.\label{p8b}
		\end{align}
		Since $f_3f_6=(f_1f_2)\big|_{q\to q^3}$, the right-hand side of \eqref{p8b} is
		a power series in $q^3$. Extracting the terms involving $q^{3n+1}$ and
		$q^{3n+2}$ from both sides of \eqref{p8b} yields the congruences
		\eqref{p8new1} and \eqref{p8new2}, respecctively.
	\end{proof}
	
	\subsection{Congruences for $\overline{p}_{9}(n)$}
	In this section, we establish a few congruences for the partition function $\overline{p}_{9}(n)$.
	\begin{theorem}For all $n\geq0$, $\alpha\geq1$
		\begin{align}\overline{p}_{9}(36n+33)\equiv0 \pmod{3},\label{mres1}\\
			\overline{p}_{9}(3^{2\alpha}(3
			n+j))\equiv0 \pmod{3},\,\,\text{for}\,\,j=1,2,\label{mres2}\\
			\overline{p}_{9}(3
			n)\equiv\overline{p}_{9}(3^{2\alpha}(3n)) \pmod{3}.\label{mres3}
		\end{align}
	\end{theorem}
	\begin{proof}
		Substituting $t=3$ in \eqref{res1} and thanks to \eqref{bt}, we get
		\begin{align}
			\overline{p}_{9}(3n)\equiv\frac{f_{1}f_2f_3}{f_{6}}\pmod{3}\label{res91}
		\end{align}
		Now invoking \eqref{f1f2} in the above equation, we get
		\begin{align}
			\overline{p}_{9}(9n)\equiv\frac{f_{3}^4}{f_{6}^2}\pmod{3},\label{res911}\\
			\overline{p}_{9}(9n+3)\equiv2\frac{f_1f_{3}f_6}{f_2}\pmod{3}\label{res912}\\
			\overline{p}_{9}(9n+6)\equiv\frac{f_1^2f_6^4}{f_2^2f_3^2}\pmod{3}.\label{res913}
		\end{align}
		Congruence \eqref{mres2} follows from \eqref{res911}. Equation \eqref{res911} implies that 
		\begin{align}
			\overline{p}_{9}(27n)\equiv\frac{f_1f_2f_3}{f_{6}}\pmod{3},\label{res914}
		\end{align}
		Congruence \eqref{mres3} follows from \eqref{res911} and the above equation. Invoking \eqref{f12byf33} in \eqref{res913} and extracting the terms involving $q^{2n+1}$ from both sides of the resulting equation, we get
		\begin{align}
			\overline{p}_{9}(18n+15)\equiv\frac{f_4f_6f_{12}}{f_2}\pmod{3}.\label{res915}
		\end{align}
		Congruence \eqref{mres1} follows from above equation. This completes the proof.   
	\end{proof}
	
	\subsection{Congruences from products of two lacunary theta-functions}
	The reductions leading to Theorems~\ref{thpdiss1} and \ref{thpdiss2} each
	obtain a single lacunary factor. Following the approach of \cite{Chen2023},
	several reductions in Sections~3 and 4 instead yield a product of two
	lacunary theta-functions, and the same vanishing follows, now from the
	representability of an integer by a binary quadratic form rather than from
	a single quadratic residue.
	
	In each of \eqref{p43i6}, \eqref{p49}, \eqref{p43252}, and \eqref{p824} the
	reduced form is a constant multiple of $\psi(q)\,\psi(q^4)$ or
	$\psi(q)\,f_4^3$. By the Jacobi identity used before Theorem~\ref{thpdiss2},
	$\psi(q)$ is supported on the triangular numbers $\tfrac{k(k+1)}{2}$, while
	both $\psi(q^4)$ and $f_4^3$ are supported on $\{2k(k+1):k\ge0\}$. Hence the
	coefficient of $q^N$ in either product is nonzero only if $N=a+4b$ with $a$,
	$b$ triangular, in which case $8a+1$ and $8b+1$ are odd squares and
	\begin{equation}
		8N+5=(8a+1)+4(8b+1)=x^2+4y^2,\qquad x,\,y\ \text{odd}.\label{tform1}
	\end{equation}
	Since $8N+5\equiv5\pmod8$, every representation $8N+5=x^2+4y^2$ has $x$ and
	$y$ odd, and such a representation exists if and only if $8N+5$ is a sum of
	two squares. By Fermat's two-square theorem this fails whenever some prime
	$p\equiv3\pmod4$ divides $8N+5$ to an odd power. If $p\mid(8b+5)$ but
	$p^2\nmid(8b+5)$, then writing $N=p^2n+b$ gives
	$8N+5=p\bigl(8pn+(8b+5)/p\bigr)$ with the second factor prime to $p$, so
	$v_p(8N+5)=1$ for every $n\ge0$. Thus the coefficient of $q^N$ in the product
	vanishes on the whole progression $N\equiv b\pmod{p^2}$.
	
	\begin{theorem}\label{thprodA}
		Let $p\equiv3\pmod4$ be prime and let $b$, $0\le b<p^2$, satisfy
		$p\mid(8b+5)$ and $p^2\nmid(8b+5)$. Then for every $n\ge0$,
		\begin{align}
			\overline{p}_{4}\bigl(64(p^2n+b)+40\bigr)&\equiv0\pmod{32},\\
			\overline{p}_{4}\bigl(32(p^2n+b)+20\bigr)&\equiv0\pmod{128},\\
			\overline{p}_{8}\bigl(8(p^2n+b)+5\bigr)&\equiv0\pmod{16},\\
			\overline{p}_{8}\bigl(16(p^2n+b)+10\bigr)&\equiv0\pmod{32}.
		\end{align}
	\end{theorem}
	\begin{proof}
		From the discussion above, the coefficient of $q^N$ in each of \eqref{p43i6},
		\eqref{p49}, \eqref{p43252}, and \eqref{p824} vanishes on
		$N\equiv b\pmod{p^2}$. Reading this class back through the outer progressions
		$64N+40$, $32N+20$, $8N+5$, and $16N+10$ of those four equations gives the
		four congruences.
	\end{proof}
	For example, $p=3$ gives $b\in\{2,8\}\pmod9$, so for every $n\ge0$,
	\begin{align}
		\overline{p}_{4}(576n+168)&\equiv0\pmod{32}, &
		\overline{p}_{4}(576n+552)&\equiv0\pmod{32},\\
		\overline{p}_{4}(288n+84)&\equiv0\pmod{128}, &
		\overline{p}_{4}(288n+276)&\equiv0\pmod{128},\\
		\overline{p}_{8}(72n+21)&\equiv0\pmod{16}, &
		\overline{p}_{8}(72n+69)&\equiv0\pmod{16},\\
		\overline{p}_{8}(144n+42)&\equiv0\pmod{32}, &
		\overline{p}_{8}(144n+138)&\equiv0\pmod{32}.
	\end{align}
	
	In \eqref{p814}, \eqref{p43231}, and \eqref{p825} the reduced form is a
	constant multiple of $\psi(q)\,f_2^3$. Since $f_2^3$ is supported on
	$\{k(k+1):k\ge0\}$, twice the triangular numbers, the coefficient of $q^N$ is
	nonzero only if $N=a+2b$ with $a$, $b$ triangular, in which case
	\begin{equation}
		8N+3=(8a+1)+2(8b+1)=x^2+2y^2,\qquad x,\,y\ \text{odd}.\label{tform2}
	\end{equation}
	Since $8N+3\equiv3\pmod8$ and squares modulo $8$ lie in $\{0,1,4\}$, every
	such representation has $x$ and $y$ odd, and $8N+3$ is represented by
	$x^2+2y^2$ if and only if every prime $p\equiv5,7\pmod8$ in its factorization
	occurs to an even power. The valuation argument of Theorem~\ref{thprodA}
	follows identically: if $p\mid(8b+3)$ but $p^2\nmid(8b+3)$, then
	$v_p(8N+3)=1$ on $N\equiv b\pmod{p^2}$, so the coefficient vanishes there.
	
	\begin{theorem}\label{thprodB}
		Let $p\equiv5$ or $7\pmod8$ be prime and let $b$, $0\le b<p^2$, satisfy
		$p\mid(8b+3)$ and $p^2\nmid(8b+3)$. Then for every $n\ge0$,
		\begin{align}
			\overline{p}_{8}\bigl(32(p^2n+b)+12\bigr)&\equiv0\pmod{16},\\
			\overline{p}_{8}\bigl(8(p^2n+b)+3\bigr)&\equiv0\pmod{32},\\
			\overline{p}_{8}\bigl(16(p^2n+b)+6\bigr)&\equiv0\pmod{32}.
		\end{align}
	\end{theorem}
	\begin{proof}
		From the discussion above, the coefficient of $q^N$ in each of \eqref{p814},
		\eqref{p43231}, and \eqref{p825} vanishes on $N\equiv b\pmod{p^2}$. Reading
		this class back through the outer progressions $32N+12$, $8N+3$, and $16N+6$
		gives the three congruences.
	\end{proof}
	For example, $p=5$ gives $b\in\{4,14,19,24\}\pmod{25}$, so for every
	$n\ge0$,
	\begin{align}
		\overline{p}_{8}(800n+140)&\equiv0\pmod{16}, &
		\overline{p}_{8}(800n+460)&\equiv0\pmod{16},\\
		\overline{p}_{8}(800n+620)&\equiv0\pmod{16}, &
		\overline{p}_{8}(800n+780)&\equiv0\pmod{16},\\
		\overline{p}_{8}(200n+35)&\equiv0\pmod{32}, &
		\overline{p}_{8}(200n+115)&\equiv0\pmod{32},\\
		\overline{p}_{8}(200n+155)&\equiv0\pmod{32}, &
		\overline{p}_{8}(200n+195)&\equiv0\pmod{32},\\
		\overline{p}_{8}(400n+70)&\equiv0\pmod{32}, &
		\overline{p}_{8}(400n+230)&\equiv0\pmod{32},\\
		\overline{p}_{8}(400n+310)&\equiv0\pmod{32}, &
		\overline{p}_{8}(400n+390)&\equiv0\pmod{32}.
	\end{align}
	
	\section{Closing Remarks}
	We close with several observations and conjectures suggested by our
	computations.
	
	\begin{conjecture}
		For all $\ell\equiv0,1\pmod4$ and $n\ge0$,
		\[
		\overline{p}_{\ell}(4n+3)\equiv0\pmod8.
		\]
		The case $\ell\equiv0\pmod8$ follows from Corollary~\ref{c1}. The remaining
		cases are supported by numerical evidence and left open. 
	\end{conjecture}
	
	\noindent Our computations further indicate the congruences
	\begin{align}
		\overline{p}_{4}(64n+56)&\equiv0\pmod{64},\\
		\overline{p}_{4}(32n+20)&\equiv0\pmod{32},\\
		\overline{p}_{4}(36n+21)&\equiv0\pmod{9},
	\end{align}
	whose proofs we leave to the interested reader.
	
	\noindent In Theorems~\ref{th6} and \ref{61} we observe that a few of the congruences
	modulo $8$ appear to hold to a higher modulus. We state this as a conjecture.
	\begin{conjecture}
		For all $n\ge0$,
		\begin{align}
			\overline{p}_{8}(32n+28)&\equiv0\pmod{32},\\
			\overline{p}_{6}(24n+17)&\equiv0\pmod{64},\\
			\overline{p}_{6}(24n+23)&\equiv0\pmod{64}.
		\end{align}
	\end{conjecture}
	
	\noindent Finally, we record some conjectured arithmetic properties of
	$\overline{p}_{9}(n)$.
	\begin{conjecture}
		For all $n\ge0$ and $\alpha\ge1$,
		\begin{align}
			\overline{p}_{9}\bigl(2^{\alpha}(8n+7)\bigr)&\equiv0\pmod{16},\\
			\overline{p}_{9}(36n+33)&\equiv0\pmod{128}.
		\end{align}
	\end{conjecture}
	
	\noindent \textbf{Funding}:
	The author(s) received no financial support for the research, authorship, and/or publication of this article.
	
\end{document}